\documentclass[12pt]{article}
\usepackage{amsfonts,amssymb,amsbsy,amsmath,amsthm,enumerate,verbatim}
\usepackage{color}

\newtheorem{theorem}{Theorem}[section]

\newtheorem{lemma}[theorem]{Lemma}

\newtheorem{pr}[theorem]{Proposition}
\theoremstyle{definition}

\newcommand{\bel}{\begin{equation} \label}
\newcommand{\ee}{\end{equation}}

\newcommand{\pd}{\partial}
\newcommand{\pf}{\partial_{\varphi}}

\newcommand{\C}{{\mathbb C}}
\newcommand{\R}{{\mathbb R}}

\newcommand{\rg}{{\rangle_g}}
\newcommand{\nag}{{\nabla_g}}
\newcommand{\lag}{{\Delta_g}}
\newcommand{\divg}{{\text{div}}_g}
\newcommand{\dvg}{{\text{dv}}_g}
\newcommand{\dwg}{{\text d} \omega_g}
\newcommand{\png}{{\partial_{\nu}}}

\newcommand{\dd}{{\rm d}}

\newcommand{\MM}{{\mathcal M}}

\newcommand{\mz}{{\mathfrak{z}}}

\newcommand{\eps}{{\varepsilon}}
\def\beq{\begin{equation}}
\def\eeq{\end{equation}}
\newcommand{\bea}{\begin{eqnarray}}
\newcommand{\eea}{\end{eqnarray}}
\newcommand{\beas}{\begin{eqnarray*}}
\newcommand{\eeas}{\end{eqnarray*}}
\newcommand{\Pre}[1]{\ensuremath{\mathrm{Re} \left( #1 \right)}}
\newcommand{\Pim}[1]{\ensuremath{\mathrm{Im} \left( #1 \right)}}
{

\begin{document}
\begin{center}
{\large \bf Identification of the twisting function for the dynamic Schr\"odinger operator in a quantum waveguide}

\medskip

\end{center}

\medskip

\begin{center}
{\sc Mourad Bellassoued, Michel Cristofol, Eric Soccorsi}\\
\end{center}

\begin{abstract}
\end{abstract}



\section{Introduction}

\subsection{Statement and origin of the problem}

\paragraph{The problem.}
\label{sec-settings}
Let $\omega$ be a bounded open subset of $\R^2$ with ${\rm C}^2$-boundary. Furthermore we assume that $\omega$ contains the origin of $\R^2$. Set $\Omega := \omega \times \R$. For $x:=(x_1,x_2,x_3) \in \Omega$ we write $x:=(x_{\tau},x_3)$ with $x_{\tau}:=(x_1,x_2)$.  Put 
$$\nabla_{\tau}:=( \partial_{1}, \partial_{2})^T,\ \Delta_{\tau} := \partial_{1}^2 + \partial_{2}^2,\ 
\pf := x_1 \partial_2 - x_2 \partial_1.
$$
We assume that $\theta = \theta(x_3) \in C^1(\R)$ and denote by $\dot{\theta}$ the derivative of $\theta$ wrt to $x_3$. Define
$H_{\dot{\theta}}$ as the self-adjoint operator in ${\rm L}^2(\Omega)$
\bel{a0}
H_{\dot{\theta}} := -\Delta_{\tau} -( \dot{\theta}(x_3) \pf + \partial_3)^2.
\ee
with Dirichlet boundary conditions on $\Gamma :=\partial \Omega$.

Given $T>0$, $q_0 \in $ and $h \in $, we then consider the Schr\"odinger equation
\bel{a1}
\left\{  \begin{array}{ll} -\imath q'(t,x) +  H_{\dot{\theta}} q(t,x) = 0, & (t,x) \in (0,T) \times \Omega  \\ q(t,x) = h(t,x), & (t,x) \in (0,T) \times \Gamma  \\ q(0,x) = q_0(x) & x \in \Omega, \end{array} \right.
\ee
where the $q'$ stands for $\frac{\partial q}{\partial t}$.

\paragraph{Twisted waveguide.}
Define the twisted domain 
$$ \Omega_{\theta} := \{ r_{\theta}(x_3) (x),\ x \in \Omega \}, $$
where
$$ r_{\theta}=r_{\theta}(x_3) := \left( \begin{array}{ccc}
\cos \theta(x_3) & \sin \theta(x_3) & 0 \\
-\sin \theta(x_3) & \cos \theta(x_3) & 0 \\
0 & 0& 1 \end{array} \right), $$
and introduce the transform
\[ (\mathcal{U}_{\theta} \psi)(x) := \psi( r_{\theta} (x) ),\ x \in \Omega,\ \psi \in {\rm L}^2(\Omega_{\theta}), \]
which is unitary from ${\rm L}^2(\Omega_{\theta})$ onto ${\rm L}^2(\Omega)$. Then we have $H_{\dot{\theta}} = \mathcal{U}_{\theta}  ( - \Delta ) \mathcal{U}_{\theta}^*$ where $\Delta$ denotes the Laplacian operator in ${\rm L}^2(\Omega_{\theta})$.
Moreover if $q$ is the solution to \eqref{a1} then $u := \mathcal{U}_{\theta}^* q$ obeys the system
$$
\left\{  \begin{array}{ll} -\imath u'(t,x) - \Delta u(t,x) = 0, & (t,x)  \in (0,T) \times \Omega_{\theta}  \\ u(t,x) = (\mathcal{U}_{\theta}^* h(t,.))(x), & (t,x) \in  (0,T) \times \partial \Omega_{\theta}  \\ u(0,x) = (\mathcal{U}_{\theta}^* q_0)(x), & x \in \Omega_{\theta}. \end{array} \right.
$$
Hence \eqref{a1} describes, modulo a unitary transform, the quantum motion of a \footnote{Here the various physical constants are taken equal to $1$.}{charged particle} in the twisted waveguide $\Omega_{\theta}$. 

In this paper we examine the inverse problem of determining the second derivative $\ddot{\theta}$ of the twisting function $\theta$ from some suitable localized, either interior or boundary, measurement of the solution $q$ to \eqref{a1}.

\subsection{Main results}
Set $I_d :=(-d,d)$ and $\Omega(d):= \omega \times I_d$ for all $d>0$.
For $\ell>0$ and $\eps>0$ we define the set of admissible twistings as
$$ 
\Theta_{\ell,\eps} = \{ \theta \in {\rm C}_0^2(\R),\ \dot{\theta} \in {\rm C}_0(I_{\ell})\ {\rm and}\ \| \dot{\theta} \|_{{\rm C}^{1}(I_{\ell})} \leq \eps \}.
$$

The first result establishes Lipschitz stability in the inverse problem of identifying $\ddot{\theta}$ from the measurement of $q$ and its time derivative $q'$ in $\Omega_{0,L}:=\omega_0 \times I_L$ for some $L>\ell$. This claim is valid for any arbitrarily small $\omega_0$ with positive $\R^2$-Lebesgue measure such that
\bel{c0d}
(0,0) \notin \omega_0
\ee
and
\bel{c0e}
\partial \omega_0 \cap \partial \omega = \emptyset.
\ee
\begin{theorem}
\label{thm1}
Let $T>0$, $\ell>0$, $\omega$ and $\Omega$ be the same as in \S \ref{sec-settings}. Let $\omega_0$ fulfill \eqref{c0d}-\eqref{c0e} and pick $\tilde{q}_0 \in {\rm H}^2(\Omega) \cap {\rm H}_0^1(\Omega)$ in such a way that
\bel{a2}
\exists \mathfrak{q}>0,\ | \pf \tilde{q}_0(x) | \geq \mathfrak{q},\ {\rm a.e.}\ x \in \Omega_0:=\omega_0 \times \R.
\ee
Let $q_0 \in {\rm H}^2(\Omega) \cap {\rm H}_0^1(\Omega)$, let $\theta$ (resp. $\tilde{\theta}$) belong to $\Theta_{\eps}$ and note $q$ (resp. $\tilde{q}$) the solution to \eqref{a1} (resp. to \eqref{a1} where $\tilde{\theta}$ and $\tilde{q}_0$ are substituted for $\theta$ and $q_0$).
Then, for every $L> \ell$ there exists a constant $C>0$, depending only on $L,T,\ell,\eps$ and $\omega_0$, such that
$$\|\ddot{\theta}-\ddot{\tilde{\theta}}\|^2_{{\rm L}^2(I_{\ell})} \leq C \left(   \| q'-\tilde{q}' \|_{{\rm L}^2(0,T;{\rm H}^1(\Omega_0(L)))}^2  + \| q_0-\tilde{q}_0 \|_{{\rm H}^2(\Omega_0)}^2 \right).$$
\end{theorem}

In contrast with Theorem \ref{thm1} only boundary, instead of volume internal measurement of the solution are required for the following result to hold.

\begin{theorem}
\label{thm2}
Let $T>0$, $L>0$, $\omega$ and $\Omega$ be the same as in \S \ref{sec-settings}. Pick $\tilde{q}_0 \in {\rm H}^2(\Omega)$ obeying \eqref{a2} for some open subset $\omega_0 \subsetneq \omega$ fulfilling $(C_2)$.
Let $\theta$ (resp. $\tilde{\theta}$) belong to $\Theta_{\eps}$ and note $q$ (resp. $\tilde{q}$) the solution to \eqref{a1} (resp. to \eqref{a1} where $\tilde{\theta}$ and $\tilde{q}_0$ are substituted for $\theta$ and $q_0$).
Then there is a constant $C'>0$, depending only on $T,L,\omega_0$ and $\ell$, such that 
$$\|\ddot{\theta}-\ddot{\tilde{\theta}}\|^2_{{\rm L}^2(I_L)} \leq C' \left(  \| \pd_{\nu}^{(g_0)} (q'- \tilde{q}') \|_{{\rm L}^2(0,T;{\rm L}^2(\Gamma_0))}^2  + \| q_0-\tilde{q}_0 \|_{{\rm H}^2(\Omega)}^2 \right).$$
\end{theorem}

\subsection{Outline}
The derivation of Theorems \ref{thm1} and \ref{thm2} is based on a global Carleman estimate for the operator $H_{\dot{\theta}}$. A Carleman inequality for the dynamic Schr\"odinger operator defined in a compact manifold $\MM \subset \R^n$, $n \geq 2$, with smooth metric $g$, is established in Proposition \ref{pr-C} in \S \ref{sec-IC}. Its proof being quite lengthy it is postponed to Appendix A. The inverse problem of determining the twisting function (actually the second derivative of the twisting function) is addressed in \S \ref{sec-IP}, which contains the proofs of Theorems \ref{thm1}-\ref{thm2}.

\section{Carleman inequality for heterogeneous Schr\"odinger operators}
\label{sec-IC}
Let $n \geq 2$ and $\MM$ be a $n$-dimensional domain with smooth boundary $\partial \MM$ and
smooth metric $g$.
This section is devoted to establishing a Carleman estimate for the dynamic Schr\"odinger operator 
\bel{c0a}
B :=\imath \partial_t + \Delta_g, 
\ee
acting in ${\rm L}^2(\MM_T)$ where $\MM_T:=(-T,T) \times \MM$.
Here $\Delta_g$ denotes the Laplace-Beltrami operator associated to the smooth metric $g$.
In local coordinates, it is given by
\bel{c0b}
\Delta_g := \frac{1}{\sqrt{{\rm det}\ g}} \sum_{j,k=1}^n \partial_j \left( \sqrt{{\rm det}\ g}\ g^{jk}\ \partial_k \right),
\ee
where $g^{-1}:=(g^{jk})_{1 \leq j,k \leq n}$ is the inverse of $g:=(g_{jk})_{1 \leq j,k \leq n}$. 

We notice from \eqref{a0} that $H_{\dot{\theta}}$ coincides with $\Delta_g$ in the particular case where $n=3$ and
\bel{c0c}
g = \left( \begin{array}{ccc} 1 + \dot{\theta}^2 x_2^2 & - \dot{\theta}^2 x_1 x_2 & - \dot{\theta} x_2 \\
- \dot{\theta}^2 x_1 x_2 & 1 + \dot{\theta}^2 x_1^2 & \dot{\theta} x_1  \\
- \dot{\theta} x_2 & \dot{\theta} x_1 & 1 \end{array} \right).
\ee

\subsection{Notations}
For each $x \in \MM$ we define the inner product
\bel{c1}
g(X,Y) = \langle X , Y \rg := \sum_{j,k=1}^n g_{j,k}(x) \alpha_j \beta_k,\ X=(\alpha_j)_{1 \leq j \leq n}, Y=(\beta_j)_{1 \leq j \leq n} \in \C^n,
\ee
and the norm $|X|_g := \langle X , \overline{X} \rg^{1 \slash 2}$.

\subsection{Carleman estimate}
Assume that there exists $\vartheta \in {\rm C}^4(\MM;\R_+^*)$ such that
\bel{caH1}
\mathbb{D}^2 \vartheta(\xi,\xi)(x) + | \langle \nag \vartheta(x) , \xi \rg |^2 >0,\ x \in \MM,\
\xi \in \R^n \setminus \{ 0 \},
\ee
where $\nag \vartheta := g^{-1} \nabla \vartheta$ and $\mathbb{D}^2$ is the Hessian of $\vartheta$ wrt the metric $g$,
and suppose that $\vartheta$ has no critical points on $\MM$,
\bel{caH2}
\min_{x \in \overline{\Omega}} | \nabla_g \vartheta(x) |_g^2 > 0.
\ee
For further reference let us notice that:
\begin{enumerate}[(i)]
\item First, \eqref{caH1} may be rephrased into the following equivalent statement
\bel{caH1b}
\mathbb{D}^2 \vartheta(\xi,\overline{\xi})(x) + | \langle \nag \vartheta(x) , \xi \rg |^2 >0,\ x \in \MM,\
\xi \in \C^n \setminus \{ 0 \},
\ee
since $\mathbb{D}^2 \vartheta(\xi,\overline{\xi})=\mathbb{D}^2 \vartheta(\xi_R,\xi_R)+\mathbb{D}^2 \vartheta(\xi_I,\xi_I)$ where $\xi_R:=\Pre{\xi}$ and $\xi_I:=\Pim{\xi}$;
\item Second, \eqref{caH2} yields the existence of some constant $\beta>0$ satisfying
\bel{caH2b}
\min_{x \in \MM} | \nabla_g \vartheta(x) |_g^2 \geq \beta.
\ee 
\end{enumerate}
Moreover we impose on $\vartheta$ the following condition 
\bel{c3}
\vartheta(x) \geq \frac{2}{3} \sup_{x \in \MM} \vartheta(x),\ x \in \MM,
\ee
which can be fulfilled by substituting $\vartheta +C$ for $\vartheta$, for any sufficiently large real number $C$.

Further define
\bel{c4}
\psi(t,x) := \frac{{\rm e}^{\gamma \vartheta(x)}}{(T-t)(T+t)},\ \eta(t,x):=\frac{{\rm e}^{2 \gamma \| \vartheta \|_{\infty}}-{\rm e}^{\gamma \vartheta(x)}}{(T-t)(T+t)},\ (x,t) \in \MM_T, 
\ee
and introduce $M := {\rm e}^{-s\eta} B  {\rm e}^{s\eta}=M_1 + M_2$, where
\bea
M_1 & := & \imath \pd_{t} +  \Delta_g + s^2 | \nabla_g \eta |_g^2 \label{c5} \\
M_2 & := & \imath s \pd_{t} \eta  + 2s \nabla \eta \cdot \nabla_g + s (\Delta_g \eta). \label{c6}
\eea

Next we consider a subboundary $\Gamma_0 \subset \partial \MM$ satisfying
\bel{caH3}
\{ x \in \Gamma, \langle \nabla_g \vartheta(x) , \nu(x) \rangle_g  \geq 0 \} \subset \Gamma_0, 
\ee
where $\nu$ denotes the outward unit normal to $\pd \MM$ wrt the metric $g$. Finally we define $\png w := \nabla_g w \cdot \nu$ on $\Sigma_0:=(-T,T) \times \Gamma_0$.

\begin{pr}
\label{pr-C}
Assume  \eqref{caH1}-\eqref{caH2} and \eqref{caH3}. Let $\psi$, $\eta$ be given by \eqref{c4}, and $M_j$, $j=1,2$, be defined by \eqref{c5}-\eqref{c6}. 
Then there exist $s_0>0$ and $C_0=C_0(\MM,T,\| g \|_{{\rm C}^2(\MM)} )>0$ such that the following Carleman estimate  
\beas
&&  I(w) :=  s \| {\rm e}^{-s \eta}  \psi^{1 \slash 2} \nabla_g w \|_{{\rm L}^2(\MM_T)^n}^2 +  s^3 \| {\rm e}^{-s \eta} \psi^{3 \slash 2} w \|_{{\rm L}^2(\MM_T)}^2 + \sum_{j=1,2} \| M_j e^{-s \eta} w \|_{{\rm L}^2(\MM_T)}^2 \nonumber\\
& \leq & C_0 \left( \| {\rm e}^{-s \eta} B w \|_{{\rm L}^2(\MM_T)}^2 
+ s \| {\rm e}^{-s \eta} \psi \png w \|_{{\rm L}^2(\Sigma_0)}^2 \right),
\eeas
holds true for every $s \geq s_0$ whenever $w \in {\rm L}^2(-T,T;{\rm H}_0^1(\MM))$ and $B w \in {\rm L}^2(\MM_T)$.
\end{pr}

\subsection{Twisted waveguide}
For $L \in (\ell,+\infty)$ fixed, set $r:=(\ell+L) \slash 2$ and let $\MM$ be any smooth bounded open subset of $\Omega(L)$ containing $\Omega(r)$:
$$ \Omega(\ell) \subset \Omega(r) \subset \MM \subset \Omega(L). $$ 
Pick $a=(a_1,a_2,a_3) \in \R^3 \setminus \MM$ such that 
\bel{c20}
d_3=a_3 - L>0, 
\ee
and put 
\bel{c21}
\vartheta(x) = | x - a |^2,\ x \in \MM. 
\ee
For further reference we notice from \eqref{c4} and \eqref{c20}-\eqref{c21} that $\eta_0=\eta(0,.)$ fulfills
\bel{c22}
\eta_0'(x) = 2 \frac{(a_3-x_3){\rm e}^{|x-a|^2}}{T^2} \geq \frac{2 d_3}{T^2},\ x \in \MM.
\ee
Moreover, in the framework of Theorem \ref{thm2} we shall impose in addition to \eqref{c20} that $a_{\tau}=(a_1,a_2) \in \R^2 \backslash \overline{\omega}$ be taken far enough from $\omega$ in such a way that 
\bel{c23}
d_{\tau} = {\rm dist}(a_{\tau},\omega)=\inf_{x_{\tau} \in \omega} |x_{\tau} - a_{\tau}|
\ee
is so large relative to $L$ and $d_3$ that
\bel{c24} 
\omega \cap B \left( a_{\tau},d_{\tau} + \frac{4L(L+d_3)}{d_{\tau}} \right) \subsetneq \omega. 
\ee 
 
By \cite{BEL1}[Lemma 4.1] we have
$4 \mathbb{D}^2\vartheta(g^{-1} \xi, g^{-1}\xi) =  \{ \mathfrak{g}, \{\mathfrak{g},\vartheta \} \}(x,\xi)$, where 
$\mathfrak{g}(x,\xi):=\sum_{j,k=1}^n g^{j,k}(x) \xi_j \xi_j$ for $x \in \MM$ and $\xi \in \R^n$. This, combined with \eqref{c0c} and the fact that
$g^{jk}=\delta_{jk}+\dot{\theta}(x_3)b^{jk}$, $1 \leq j,k \leq 3$, yields
$$ \mathbb{D}^2 \vartheta(g^{-1}\xi,g^{-1}\xi)=\mathrm{Hess}(\vartheta)(\xi,\xi)+\mathcal{O}(\| \dot{\theta}\|_{\text{C}^1(I_L)} | \xi |^2),$$
where $\mathrm{Hess}(\vartheta)$ denotes the Hessian of $\vartheta$ wrt the Euclidian metric. Therefore we have
$$
\mathbb{D}^2 \vartheta(g^{-1} \xi,g^{-1} \xi)=2 | \xi |^2+\mathcal{O}(\| \dot{\theta} \|_{\text{C}^1(I_{\ell})} | \xi |^2),\ \xi \in \R^n,
$$
by an elementary calculation, whence $\mathbb{D}^2 \vartheta(\xi,\xi) \geq C | \xi |^2$ for all $\xi \in \R^n$ and some constant $C>0$, upon taking $\| \dot{\theta} \|_{\text{C}^1(I_{\ell})}$ sufficiently small. Namely, there exists $\eps>0$, depending only on $\omega$, $\ell$ and $a$ such that \eqref{caH1} holds true for every $\theta \in \Theta_{\ell,\eps}$.

\section{Inverse problem}
\label{sec-IP}
In this section we investigate the inverse problem of identifying $\ddot{\theta}$ from partial knowledge of the solution $q$ to \eqref{a1}. We shall prove that $\ddot{\theta}$ can be determined uniquely and stably from some suitable localized, either interior or boundary, measurement of $q$, as stated in Theorem \ref{thm1} and Theorem \ref{thm2}.
In both cases the method used in the derivation of the corresponding stability estimate preliminarily requires that the problem under study be centered around $t=0$, then appropriately ``linearized" and differentiated w.r.t. the time variable $t$.

%
\subsection{Time symmetrization and the linearized problem}
\paragraph{Time symmetrization.}
Following the idea of \cite{BP} we extend \eqref{a1} to negative values of the time variable $t$ in order to center the problem around $t=0$. 
Indeed, for any real valued initial data $q_0$, the solution $q$ to \eqref{a1} extended to $Q := (-T,T) \times \Omega$ by setting
\bel{b0}
q(t,x):=\overline{q(-t,x)}\ {\rm for\ a.e.}\ (t,x)  \in \R_-^* \times \Omega,
\ee
is easily seen to be solution to the system
\bel{b1}
\left\{  \begin{array}{ll} -\imath q'(t,x) +  H_{\dot{\theta}} q(t,x) = 0, & (t,x) \in Q 
  \\ q(t,x) = h(t,x), & (t,x) \in \Sigma := (-T,T) \times \Gamma  \\ q(0,x) = q_0(x), & x \in \Omega, \end{array} \right.
\ee
provided we have preliminarily set
$$ h(t,x):=\overline{h(-t,x)}\ {\rm for\ a.e.}\ (t,x)  \in \R_-^* \times \Gamma. $$

\paragraph{The ``linearized" problem.} Let $\tilde{q}$ denote the solution to \eqref{b1} where
$\theta$ and $q_0$ are replaced by $\tilde{\theta}$ and $\tilde{q}_0$, a straightforward computation shows that $y:=q-\tilde{q}$ satisfies the ``linearized" system 
\bel{b3}
\left\{  \begin{array}{ll} -\imath  y'(t,x) +  H_{\dot{\theta}} y(t,x) = R(t,x), & (t,x) \in Q  \\ y(t,x) = 0, & (t,x) \in \Sigma \\ y(0,x) = (q_0- \tilde{q}_0)(x), & x \in \Omega, \end{array} \right.
\ee
where
\bel{b3star}
\alpha := \dot{\theta}-\dot{\tilde{\theta}}\ {\rm and}\ R := (H_{\dot{\tilde{\theta}}}-H_{\dot{\theta}}) \tilde{q}= (\alpha ((\dot{\theta}+\dot{\tilde{\theta}}) \pf + 2  \pd_3 ) + \dot{\alpha} ) \pf \tilde{q}.
\ee
Further, by differentiating \eqref{b3} w.r.t. $t$, we get that $z := y'$ is solution to
\bel{b4}
\left\{  \begin{array}{ll} -\imath z'(t,x) +  H_{\dot{\theta}} z(t,x) = R'(t,x), & (t,x) \in Q \\ z(t,x) = 0, & (t,x) \in \Sigma \\ z(0,x) = \imath ( R(0,x) - H_{\dot{\theta}} y(0,x) ), & x \in \Omega, \end{array} \right.
\ee
with
\bel{b4b}
R' :=  (\alpha ((\dot{\theta}+\dot{\tilde{\theta}}) \partial_{\varphi} + 2  \partial_3 ) + \dot{\alpha} ) \pf \tilde{q}',
\ee
according to \eqref{b3star}.

\subsection{Proof of Theorem \ref{thm1}}
Let $\omega_1$ be any non empty open subset of $\omega_0$ and pick two cut-off functions
$\rho \in {\rm C}^2(\omega;[0,1])$ and $\mu \in {\rm C}^2(\R;[0,1])$ satisfying
\bel{b4c}
\rho(x_{\tau}) := \left\{ \begin{array}{ll} 1 & {\rm in}\ \omega_1  \\ 0 & {\rm in}\  \omega \backslash \omega_0 \end{array} \right.\
{\rm and}\
\mu(x_3) := \left\{ \begin{array}{ll} 1 & {\rm in}\ I_{\ell}  \\ 0 & {\rm in}\  \R \backslash I_{r}, \end{array} \right.
\ee
in such a way that $\chi:=\rho \mu \in {\rm C}^{2}(\Omega;[0,1])$ fulfills
\bel{b4d}
\chi(x) = \left\{ \begin{array}{ll} 1 & {\rm in}\ \Omega_{1}(\ell)  \\ 0 & {\rm in}\  \Omega \setminus \Omega_{0}(r). \end{array} \right.
\ee
By left multiplying \eqref{b4} by $\chi$, we get that $\mz = \chi z$ is solution of the system 
\bel{b5}
\left\{  \begin{array}{ll} -\imath \mz'(t,x) +  H_{\dot{\theta}} \mz(t,x) = \chi R'(t,x) + [H_{\dot{\theta}}, \chi] z, & (t,x) \in \MM_T \\ \mz(t,x) = 0, & (t,x) \in \pd \MM_T := (-T,T) \times \pd \MM \\ \mz(0,x) = -\imath \chi (R(0,x)- H_{\dot{\theta}} y(0,x)), & x \in \MM. \end{array} \right.
\ee
From this and Proposition \ref{pr-C} then follows that
\beas 
&&  I(\mz) = s \| {\rm e}^{-s \eta}  \psi^{1 \slash 2} \nabla_{g} \mz \|_{{\rm L}^2(\MM_T)^3}^2 +  s^3 \| {\rm e}^{-s \eta} \psi^{3 \slash 2} \mz \|_{{\rm L}^2(\MM_T)}^2 + \sum_{j=1,2} \| M_j e^{-s \eta} \mz \|_{{\rm L}^2(\MM_T)}^2 \\
& \leq & C_0 \| {\rm e}^{-s \eta} ( \chi R'(t,x) + [H_{\dot{\theta}}, \chi] z ) \|_{{\rm L}^2(\MM_T)}^2,\ s \geq s_0.
\eeas
In light of \eqref{b4b} and the expression $[H_{\dot{\theta}},\chi] = -( 2 \nabla_{g} \chi \cdot \nabla + \Delta_{g} \chi )$, this yields
\bel{b6}
I(\mz) \leq  C_1  \sum_{m=0,1}  \left( \| {\rm e}^{-s \eta} \chi \alpha^{(m)} \|_{{\rm L}^2(-T,T;{\rm L}^2(\Omega_0(\ell)))}^2 +
\| {\rm e}^{-s \eta} \nabla^m z \|_{{\rm L}^2(-T,T;{\rm L}^2(\Omega_{0}(r) \setminus \Omega_1(\ell)))}^2 \right),\ s \geq s_0,
\ee
for some constant $C_1=C_1(\omega,L,T,\eps)>0$, where $\alpha^{(0)}$ (resp. $\alpha^{(1)}$) stands for $\alpha$ (resp. $\dot{\alpha}$). \\
Having established the Carleman estimate \eqref{b6}, the remaining of the proof now decomposes into four steps.

The first step involves relating $\dot{\alpha}$ to $\mz(0,.)$ with the aid of \eqref{b5}.
Namely, by multiplying the initial condition in \eqref{b5} by ${\rm e}^{-s \eta_0}$, squaring and integrating over $\MM$,
we deduce from \eqref{b3star} that 
$\| {\rm e}^{-s \eta_0} \pf \tilde{q}_0 \chi \dot{\alpha} \|_{{\rm L}^2(\Omega_0(\ell))}^2$ is majorized (up to some positive multiplicative constant $C_2$ depending only on $\omega$, $L$, $T$ and $\eps$) for all $s>0$ by the sum
$ \| {\rm e}^{-s \eta_0} \mz(0,.) \|_{{\rm L}^2(\MM)}^2  + \| {\rm e}^{-s \eta_0} H_{\dot{\theta}} y(0,.) \|_{{\rm L}^2(\Omega_0(r))}^2 +  \| {\rm e}^{-s \eta_0} \chi \alpha \|_{{\rm L}^2(\Omega_0(\ell))}^2$.
Taking into account that $| \pf \tilde{q}_0 | \chi \geq \mathfrak{q} \chi$, this yields
\bea
& & \| {\rm e}^{-s \eta_0} \chi \dot{\alpha} \|_{{\rm L}^2(\Omega_0(\ell))}^2 \nonumber \\
& \leq & C_2 \left( \| {\rm e}^{-s \eta_0} \mz(0,.) \|_{{\rm L}^2(\MM)}^2  + \| {\rm e}^{-s \eta_0} H_{\dot{\theta}} y(0,.) \|_{{\rm L}^2(\Omega_0(r))}^2 + 
\| {\rm e}^{-s \eta_0} \chi \alpha \|_{{\rm L}^2(\MM)}^2 \right), \label{b7}
\eea
where $C_2=C_2(\omega,L,T,\eps)>0$.
The next step is to derive a suitable upper bound on $\| {\rm e}^{-s \eta_0} \mz(0,.) \|_{{\rm L}^2(\MM)}$ involving $\| {\rm e}^{-s \eta} \chi \dot{\alpha} \|_{{\rm L}^2(\Omega_0(\ell))}$, from \eqref{b6}.

\begin{lemma}
\label{lm1}
There exists a constant $C_3=C_3(\omega,L,T,\eps)>0$ such that we have
\beas
& & \| {\rm e}^{-s \eta_0} \mz(0,.) \|_{{\rm L}^2(\MM)}^2 \nonumber \\
& \leq & C_3  s^{-3/2}  \sum_{m=0,1} \left( \| {\rm e}^{-s \eta} \chi \alpha^{(m)} \|_{{\rm L}^2(-T,T;{\rm L}^2(\Omega_0(\ell)))}^2 + \| {\rm e}^{-s \eta} \nabla^m z \|_{{\rm L}^2(-T,T;{\rm L}^2(\Omega_{0}(r) \setminus \Omega_1(\ell)))}^2 \right),
\eeas
for all $s \geq s_0$.
\end{lemma}
\begin{proof}
Set $\zeta := {\rm e}^{-s \eta} \mz$ and recall that $\zeta(-T,.)=0$ so that
$$
\mathcal{I}  :=  \| {\rm e}^{-s \eta_0} \mz(0,.) \|_{{\rm L}^2(\MM)}^2 = \int_{-T}^{0} \int_{\MM} \pd_{\tau} |\zeta(t,x)|^2 {\rm d} x {\rm d} t = 2\Pre{\int_{-T}^{0} \int_{\MM} \zeta_{\tau}(t,x) \overline{\zeta(t,x)} {\rm d} x {\rm d} t}.
$$
In view of \eqref{c5} this identity may be rewritten as
$$ \mathcal{I} = 2 \Pim{\int_{-T}^0 \int_{\MM} (M_1 \zeta) \overline{\zeta}(t,x) {\rm d} t {\rm d} x}, $$
so we get
$|\mathcal{I}| \leq 2 \| M_1 {\rm e}^{-s \eta} \mz \|_{{\rm L}^2(\MM_T)} \| {\rm e}^{-s \eta} \mz \|_{{\rm L}^2(\MM_T)}$ from the Cauchy-Schwarz inequality. As $\inf_{(t,x) \in Q} \psi(t,x) >0$ according to \eqref{c4}, there is necessarily a constant $c>0$, depending only on $T$, such that we have
$$
|\mathcal{I}| \leq c s^{-3 \slash 2}  \left( s^3  \| {\rm e}^{-s \eta} \psi^{3 \slash 2}  \mz \|_{{\rm L}^2(\MM_T)}^2  + \| M_1  {\rm e}^{-s \eta} \mz \|_{{\rm L}^2(\MM_T)}^2 \right) \leq c s^{-3 \slash 2} I(\mz),\ s>0.
$$
This, combined with \eqref{b6}, proves the result.
\end{proof}

The third step consists of showing that $\| {\rm e}^{-s \eta_0} \chi \alpha \|_{{\rm L}^2(\MM)}$ can be made negligible w.r.t. $\| {\rm e}^{-s \eta_0} \chi \dot{\alpha} \|_{{\rm L}^2(\MM)}$ upon taking $s$ sufficiently large.

\begin{lemma}
\label{lm2}
We have
$$ \| {\rm e}^{-s \eta_0} \chi \alpha \|_{{\rm L}^2(\Omega_0(\ell))} \leq \frac{T}{\sqrt{d_3 s}} \| {\rm e}^{-s \eta_0} \chi \dot{\alpha} \|_{{\rm L}^2(\Omega_0(\ell))},\ s>0, $$
where $d_3$ is defined in \eqref{c20}.
\end{lemma}

\begin{proof}
Bearing in mind that $\alpha(\pm \ell)=0$, an integration by parts yields
$$ 
\int_{\Omega_0(\ell)} {\rm e}^{-2 s \eta_0} \chi^2 \alpha  \dot{\alpha}  \dd x  
=\int_{\Omega_0(\ell)} {\rm e}^{-2 s \eta_0} \rho^2 \alpha  \dot{\alpha}  \dd x
= -\frac{1}{2}\int_{\Omega_0(\ell)}  \pd_{x_3} ({\rm e}^{-2 s \eta_0}) \rho^2 \alpha^2 \dd x.
$$
Therefore,
\bel{b9}
\int_{\Omega_0(\ell)} {\rm e}^{-2 s \eta_0} \chi^2 \alpha  \dot{\alpha}  \dd x
=\int_{\Omega_0(\ell)} {\rm e}^{-2 s \eta_0}  s \dot{\eta}_0 \rho^2 \alpha^2 \dd x.
=s \int_{\Omega_0(\ell)} {\rm e}^{-2 s \eta_0} \dot{\eta}_0 \chi^2 \alpha^2 \dd x.
\ee
Further, as
$$\int_{\Omega_0(\ell)} {\rm e}^{-2 s \eta_0} \chi^2 \alpha  \dot{\alpha}  \dd x
\leq \frac{d_3}{T^2} s  \int_{\Omega_0(\ell)} {\rm e}^{-2 s \eta_0} \chi^2 \alpha^2  \dd x + \frac{T^2}{d_3 s}  \int_{\Omega_0(\ell)} {\rm e}^{-2 s \eta_0} \chi^2 \dot{\alpha}^2 \dd x,\ s>0,$$
from Young inequality, \eqref{b9} yields 
$$ s \int_{\Omega_0(\ell)} {\rm e}^{-2s \eta_0} \left( \dot{\eta}_0 - \frac{d_3}{T^2} \right) \chi^2 \alpha^2  \dd x \leq  \frac{T^2}{d_3 s}  \int_{\Omega_0(\ell)} {\rm e}^{-2 s \eta_0} \chi^2 \dot{\alpha}^2 \dd x,\ s>0.$$
This, together with \eqref{c22}, completes the proof.
\end{proof}
In view of \eqref{b7} and Lemmae \ref{lm1}-\ref{lm2} we thus may find some constant $C_4=C_4(\omega,L,T,\eps)>0$ such that the following estimate
$$
\| {\rm e}^{-s \eta_0} \chi \dot{\alpha} \|_{{\rm L}^2(\Omega_0(\ell))}^2 \leq
C_4 \left( \| {\rm e}^{-s \eta_0} H_{\dot{\theta}} y(0,.) \|_{{\rm L}^2(\Omega_0(r))}^2 + \sum_{m=0,1} \| {\rm e}^{-s \eta} \nabla^m z \|_{{\rm L}^2(-T,T;{\rm L}^2(\Omega_0(r) \setminus \Omega_1(\ell)))}^2 \right),
$$
holds true upon choosing $s$ sufficiently large. Taking into account that $\| {\rm e}^{-s \eta} \|_{{\rm L}^{\infty}(\MM_T)} \leq 1$ and $\inf_{x \in \Omega_0(\ell)} {\rm e}^{-s \eta_0(x)} >0$ for any $s>0$, this entails 
\bel{b10}
\| \chi \dot{\alpha} \|_{{\rm L}^2(\Omega_0(\ell))}^2 \leq
C_5 \left( \|  H_{\dot{\theta}} y(0,.) \|_{{\rm L}^2(\Omega_0(r))}^2 + \|  z \|_{{\rm L}^2(-T,T;{\rm H}^1(\Omega_0(r) \setminus \Omega_1(\ell)))}^2 \right),
\ee
under the same conditions on $s$, where $C_5$ is some positive real number depending only on $\omega$, $L$, $T$ and $\eps$.

The last step involves noticing that
\bel{b11}
\| \chi \dot{\alpha} \|_{{\rm L}^2(\Omega_0(\ell))} \geq \mathfrak{q}\ {\rm mes}(\omega_1)
\| \dot{\alpha} \|_{{\rm L}^2(I_{\ell})},\ s >0,
\ee
since $\chi(x)=1$ for every $x \in \Omega_{1}(\ell)$ and $\dot{\alpha}(x_3)=0$ whenever $x_3 \in \R \setminus I_{\ell}$, by \eqref{a2} and \eqref{b4c}-\eqref{b4d}. 

Finally, bearing in mind that
$\|  H_{\dot{\theta}} y(0,.) \|_{{\rm L}^2(\Omega_0(r))} \leq C_6 \| y(0,.) \|_{{\rm H}^2(\Omega_0(r))}$, where the constant $C_6>0$ depends only on $\omega_0$ and $\eps$, we get the result by combining \eqref{b10} and \eqref{b11}.

\subsection{Proof of Theorem \ref{thm2}}
Set $\mz:= \mu z$, where $\mu$ denotes the cut off function defined by \eqref{b4c}. Since $\mz$ is solution to the system \eqref{b5} where $\mu$ is substituted for $\chi$, we have the following Carleman estimate
$$ I(\mz) \leq C_0 \left( \| {\rm e}^{-s \eta} ( \mu R'(t,x) + [H_{\dot{\theta}}, \mu] z ) \|_{{\rm L}^2(\MM_T)}^2
+ s  \| {\rm e}^{-s \eta}  \psi^{1 \slash 2} \partial_{\nu} \mz \|_{{\rm L}^2(\Sigma_0)}^2 \right), s \geq s_0,
$$
arising from Proposition \ref{pr-C}. Putting $\Omega(\ell,r):=\Omega(r) \setminus \Omega(\ell)$ and $\Sigma_0(r): =(-T,T) \times \Gamma_0(r)$ with
$\Gamma_0(r):= \Gamma_0 \cap  (\pd \omega \times I_r)$, there is thus a constant $C_1'>0$, depending only on, such that
we have
\bea
& & I(\mz) \nonumber \\
& \leq & C_1' \left[ \sum_{m=0,1} \left( \| {\rm e}^{-s \eta} \alpha^{(m)} \|_{{\rm L}^2(-T,T;{\rm L}^2(\Omega(\ell)))}^2 +  \| {\rm e}^{-s \eta} \nabla^m z \|_{{\rm L}^2(-T,T;{\rm L}^2(\Omega(\ell,r)))} \right) +
s \|  {\rm e}^{-s \eta} \partial_{\nu} z \|_{{\rm L}^2(\Sigma_0(r))}^2 \right], \label{b12}
\eea
for all $s \geq s_0$.

The strategy of the proof of Theorem \ref{thm1} applies unchanged provided \eqref{b6} is replaced by
\eqref{b12}.
Indeed, arguing as in the derivation of \eqref{b7} from \eqref{b5}, we find out some positive constant $C_2'=C_2'()$ satisfying
\bel{b13}
\| {\rm e}^{-s \eta_0} \pf \tilde{q}_0\ \dot{\alpha} \|_{{\rm L}^2(\Omega(\ell))}^2 \leq C_2' \left( \| {\rm e}^{-s \eta_0} \mz(0,.) \|_{{\rm L}^2(\MM)}^2  + \| {\rm e}^{-s \eta_0} H_{\dot{\theta}} y(0,.) \|_{{\rm L}^2(\Omega(r))}^2 + 
\| {\rm e}^{-s \eta_0} \alpha \|_{{\rm L}^2(\Omega(\ell))}^2 \right),
\ee
for all $s>0$. Further, by substituting \eqref{b12} for \eqref{b6} in the proof of Lemma \ref{lm1}, we find out that 
\bea
& & C_3' s^{3 \slash 2} \| {\rm e}^{-s \eta_0} \mz(0,.) \|_{{\rm L}^2(\MM)}^2 \nonumber \\
& \leq &  \sum_{m=0,1} \left( \| {\rm e}^{-s \eta_0} \alpha^{(m)} \|_{{\rm L}^2(-T,T;{\rm L}^2(\Omega(\ell)))}^2 + 
\| {\rm e}^{-s \eta} \nabla^{(m)} z \|_{{\rm L}^2(-T,T;{\rm L}^2(\Omega(\ell,r)))}^2 \right) \nonumber \\
& & \ \ \ \ \ \ \ \ \ \ \ \ + s \| {\rm e}^{-s \eta} \pd_{\nu} z  \|_{{\rm L}^2(\Sigma_0(r))}^2,
\label{b14}
\eea
for $s \geq s_0$ and some constant $C_3'=C_3'()>0$.
Last, by mimicking the proof of Lemma \ref{lm2}, we find out that
\bel{b15}
\| {\rm e}^{-s \eta_0} \alpha \|_{{\rm L}^2(\Omega(\ell))} \leq \frac{T^2}{d_3 s} \| {\rm e}^{-s \eta_0} \dot{\alpha} \|_{{\rm L}^2(\Omega(\ell))},\ s >0.
\ee
Now, upon setting
\bel{b16b}
{\rm obs} := \| {\rm e}^{-s \eta} \partial_{\nu} z \|_{{\rm L}^2(\Sigma_0(r))}^2 +  \| {\rm e}^{-s \eta_0} z \|_{{\rm L}^2(-T,T;{\rm H}^1(\Omega(\ell,r)))} + \| {\rm e}^{-s \eta_0} H_{\dot{\theta}} y(0,.)\|_{{\rm L}^2(\Omega(r))}^2,
\ee
it turns out by putting \eqref{b13}, \eqref{b14} and \eqref{b15} together, that
\bel{b16}
\| {\rm e}^{-s \eta_0}  \pf \tilde{q}_0\ \dot{\alpha} \|_{{\rm L}^2(\Omega(\ell))}^2 
\leq C_4' ( {\rm obs} + s^{-3/2} \| {\rm e}^{-s \eta_0}  \dot{\alpha} \|_{{\rm L}^2(\Omega(\ell))}^2 ), 
\ee
provided $s$ is sufficiently large. Here $C_4'$ is some positive constant depending only on $ $.
The next step of the proof involves splitting $\| {\rm e}^{-s \eta_0}  \dot{\alpha} \|_{{\rm L}^2(\Omega(\ell))}^2$ into the sum $\| {\rm e}^{-s \eta_0}  \dot{\alpha} \|_{{\rm L}^2(\Omega_0(\ell))}^2 +  \| {\rm e}^{-s \eta_0}  \dot{\alpha} \|_{{\rm L}^2(\Omega(\ell) \setminus \Omega_0(\ell))}^2$, where the subset $\Omega_0(\ell)=\omega_0 \times I_{\ell}$ is chosen in accordance with the following statement. 


\begin{lemma}
\label{lm3}
Assume \eqref{c23}-\eqref{c24}. Then there exist two proper open subsets of $\omega$,
$\omega_1  \subsetneq \omega_0$, with positive $\R^2$-Lebesgue measure, satisfying:
$$ 
m_1 = \sup_{x \in \Omega_1(\ell)} \eta_0(x) <  m_0 = \inf_{x \in \Omega(\ell) \backslash \Omega_0(\ell)} \eta_0(x). 
$$
\end{lemma}
\begin{proof}
For $\epsilon=2L(L+d_3) \slash d_{\tau}$, where $d_{\tau}$ is the same as in \eqref{c23}, we put
$$\omega_j = \{ x_{\tau} \in \omega,\ | x_{\tau} - a_{\tau} | > d_{\tau} +(j+1) \epsilon \},\ j=0,1, $$
in such a way that 
$\omega_1  \subsetneq \omega_0 \subsetneq \omega$ according to \eqref{c23}-\eqref{c24}, and
\bel{b18}
\inf_{x_{\tau} \in \omega_1} | x_{\tau}- a_{\tau} | \geq \sup_{x_{\tau} \in \omega \setminus \omega_0} | x_{\tau}- a_{\tau} | + \epsilon.
\ee
Further, as $|x_{\tau}-a_{\tau}|^2 + d_3^2 \leq \vartheta(x)  \leq |x_{\tau}-a_{\tau}|^2 + (2L+d_3)^2$ for all $x=(x_{\tau},x_3) \in \Omega(\ell)$,
by \eqref{c20}-\eqref{c21}, it follows from \eqref{b18} that 
\bel{b19}
\vartheta(x) \geq \vartheta(x') + R(x_{\tau}'),\ x \in \Omega_1(\ell),\ x'=(x_{\tau}',x_3') \in \Omega(\ell) \setminus \Omega_0(\ell),
\ee
where $R(x_{\tau}')= 2 \epsilon (|x_{\tau}' - a_{\tau}| - d_{\tau}) + \epsilon^2$. Since $|x_{\tau}'-a_{\tau}|\geq d_{\tau}$ from \eqref{c23}, then $R(x_{\tau}') \geq \epsilon^2$ and 
$$ \vartheta(x) \geq \vartheta(x') + \epsilon^2,\ x \in \Omega_1(\ell),\ x' \in \Omega(\ell) \setminus \Omega_0(\ell). $$
This entails $\inf_{x \in \Omega_1(\ell)} \vartheta(x) > \sup_{x \in \Omega(\ell) \setminus \Omega_0(\ell)} \vartheta(x)$, hence the result by \eqref{c4}.
\end{proof}
Since $\| {\rm e}^{-s \eta_0}  \pf \tilde{q}_0\ \dot{\alpha} \|_{{\rm L}^2(\Omega(\ell))} \geq \mathfrak{q}\ \| {\rm e}^{-s  \eta_0} \dot{\alpha} \|_{{\rm L}^2(\Omega_0(\ell))}$ by \eqref{a2}, we deduce from \eqref{b16}  that
$$ (\mathfrak{q} - C_4' s^{-3/2} ) \|  {\rm e}^{-s \eta_0} \dot{\alpha} \|_{{\rm L}^2(\Omega_0(\ell))}^2 \\
\leq C_4' ( {\rm obs} + s^{-3/2} \| {\rm e}^{-s \eta_0} \dot{\alpha} \|_{{\rm L}^2(\Omega \setminus \Omega_0(\ell))}^2 ).
$$
Therefore we have
\bel{b20}
\|  {\rm e}^{-s \eta_0} \dot{\alpha} \|_{{\rm L}^2(\Omega_1(\ell))}^2 
\leq C_5' ( {\rm obs} + s^{-3/2} \| {\rm e}^{-s \eta_0} \dot{\alpha} \|_{{\rm L}^2(\Omega(\ell) \setminus \Omega_0(\ell))}^2 ), 
\ee
for some $C_5'=C_5'()>0$ and all $s$ large enough. Here $\Omega_1(\ell)$ is the same as in Lemma \ref{lm3} so \eqref{b20} yields
\bel{b21}
{\rm e}^{-s m_1} ( {\rm mes}(\omega_1) - s^{-3/2} {\rm e}^{-s (m_0-m_1)} {\rm mes}(\omega) ) \| \dot{\alpha} \|_{{\rm L}^2(I_{\ell})}^2 \leq C_5'\ {\rm obs},
\ee
under the above prescribed conditions on $s$. Finally, by recalling \eqref{b16b} and noticing that
$$ {\rm obs} \leq C_6' ( \| \partial_{\nu} z \|_{{\rm L}^2(\Sigma_0(r))}^2 + \| z \|_{{\rm L}^2(-T,T;{\rm H}^1(\Omega(\ell,r)))} + \| y(0,.) \|_{{\rm H}^2(\Omega(r))}^2 ), $$
where $C_6'=C_6'()>0$, the desired result follows directly from this, \eqref{b21} and Lemma \ref{lm3} upon taking $s$ sufficiently large.

\section{Appendix : proof of Proposition \ref{pr-C}}

\subsection{Preliminaries}
Put 
$$\divg X := \frac{1}{\sqrt{{\rm det} g}} \sum_{j=1}^n \partial_j(\sqrt{{\rm det} g} \alpha_j),\ X=(\alpha_j)_{1 \leq j \leq n} \in \C^n,$$
in such a way that \eqref{c0b} reads $\Delta_g = \divg \nabla_g$. For $X$ and $f$ in ${\rm H}^1(\MM)$ we have
\bel{gdiv}
\divg( f X ) = \langle \nag f , X \rg + f \divg X.
\ee
For $X$ and $f$ in ${\rm H}^1(\MM)$ we have the following Green formula
\bel{gGreen}
\int_{\MM} \divg(X) f\ \dvg = -\int_{\MM} \langle X , \nag f \rg\ \dvg +  \int_{\pd \MM} X \cdot \nu f \dwg,
\ee
which yields
\bel{gStokes}
\int_{\MM} (\lag w) f\ \dvg = -\int_{\MM} \langle \nag w , \nag f \rg\ \dvg +  \int_{\pd \MM} (\png w) f \dwg,
\ee
whenever $f \in {\rm H}^1(\MM)$ and $w \in {\rm H}^2(\MM)$.

\subsection{Completion of the proof}

In light of \eqref{c5}-\eqref{c6} we shall first compute each term in the rhs of the following identity
\bel{ca0}
\Pre{\int_{\MM_T} M_1 w \overline{M_2 w}\ \dvg\ \dd t} = \sum_{j,k=1,2,3} I_{j,k},
\ee
separately.
We start by noticing that
\beas 
I_{1,1} & := & \Pre{\int_{\MM_T} \imath w' ( - \imath s \eta' \overline{w})\ \dvg\ \dd t}
= s \Pre{\int_{\MM_T} w' \eta' \overline{w}\  \dvg\ \dd t} \\
& = & - s \Pre{\int_{\MM_T} w ( \eta'' \overline{w} + \eta' \overline{w'})\ \dvg\  \dd t} =
-s \Pre{\int_{\MM_T} \eta'' | w |^2\ \dvg\ \dd t} - I_{1,1},
\eeas
where we used the fact that $w(\pm T)=0$. Hence
\bel{ca1}
I_{1,1} = - \frac{s}{2} \int_{\MM_T} \eta'' | w |^2 \dvg\ \dd t.
\ee
Next, bearing in mind that $w(\pm T)=0$, we have
\bea
I_{1,2} & := & 2 s \Pre{\int_{\MM_T} \imath \langle \nag \eta , \nag \overline{w} \rg w'\ \dvg\ \dd t} 
= - 2s \Pim{\int_{\MM_T} \langle \nag \eta , \nag \overline{w} \rg w'\ \dvg\ \dd t} \nonumber \\
& = & 2 s \Pim{\int_{\MM_T} \left( \langle \nag \eta' , \nag \overline{w} \rg + \langle \nag \eta , \nag \overline{w'} \rg \right) w\ \dvg\ \dd t}, \label{ca2}
\eea
with, in addition,
$$ \int_{\MM} \langle \nag \eta , \nag \overline{w'} \rg w\ \dvg = -\int_{\MM} \divg((\nag \eta) w) \overline{w'} \dvg = - \int_{\MM} ( (\lag \eta) w + \overline{\langle \nag \eta , \nag \overline{w} \rg ) w'} \dvg, $$
by taking into account that $w'$ vanishes on $\pd \MM$. Therefore, it follows from this and \eqref{ca2} that
$$
I_{1,2} = 2 s \Pim{\int_{\MM_T} \left( \langle \nag \eta' , \nag \overline{w} \rg - (\lag \eta) \overline{w'} \right) w\ \dvg\ \dd t} - I_{1,2}, $$
from where we get that
\bel{ca3}
I_{1,2} = s \gamma \Pim{\int_{\MM_T} \left( \psi (\lag \vartheta + \gamma | \nag \vartheta |_g^2) \overline{w'} - \psi' \langle \nag \vartheta , \nag \overline{w} \rg \right) w\ \dvg\ \dd t}.
\ee
Further we have
\bel{ca4}
I_{1,3} := s \Pre{\int_{\MM_T} \imath w' (\lag \eta) \overline{w}\ \dvg\ \dd t} = -s \gamma \Pim{\int_{\MM_T} \psi ( \lag \vartheta + \gamma | \nag \vartheta |_g^2) w \overline{w'}\ \dvg\ \dd t},
\ee
and $I_{2,1}:=\Pre{\int_{\MM_T} (\lag w) (-\imath s \eta' \overline{w})\ \dvg\ \dd t} = s \Pim{\int_{\MM_T} (\lag w) \eta' \overline{w}\ \dvg\ \dd t}$, where
$$
\int_{\MM} (\lag w) \eta' \overline{w}\ \dvg = - \int_{\MM} \overline{\langle \nag \eta' , \nag \overline{w} \rg w}\ \dvg
- \int_{\MM} \eta' \langle \nag w , \nag \overline{w} \rg\ \dvg, $$
since $w=0$ on $\pd \MM$, so we find that
\bel{ca5}
I_{2,1} = -s \gamma \Pim{\int_{\MM_T} \psi' \langle \nag  \vartheta , \nag \overline{w} \rg w\ \dvg\ \dd t}.
\ee
Now, putting \eqref{ca3}, \eqref{ca4} and \eqref{ca5} together we obtain that
\bel{ca5b}
I_{1,2}+I_{1,3}+I_{2,1}=-2s \gamma \Pim{\int_{\MM_T} \psi' \langle \nag \vartheta , \nag \overline{w} \rg w\ \dvg\ \dd t}.
\ee
We turn now to computing $I_{2,2}:=2 s \Pre{\int_{\MM_T} (\lag w) \langle \nag \eta , \nag \overline{w} \rg\ \dvg\ \dd t}$.
In light of \eqref{gStokes} we see that
\bel{ca6}
\int_{\MM} (\lag w) \langle \nag \eta , \nag \overline{w} \rg\ \dvg = - \int_{\MM} \langle \nag w , \nag \langle \nag \eta , \nag \overline{w} \rg \rg\ \dvg + \int_{\pd \MM} \langle \nag \eta , \nag \overline{w} \rg
\png w\ \dwg.
\ee
Bearing in mind that $w=0$ and hence $\nag w \cdot \tau=0$ on $\pd \MM$, where $\tau$ denotes the unit tangential vector to $\pd \MM$, we have
$\langle \nag \eta , \nag \overline{w} \rg =  (\nabla \eta \cdot \nu) \png \overline{\omega}$,
hence
\bel{ca7} 
\int_{\pd \MM} \langle \nag \eta , \nag \overline{w} \rg \png w\ \dwg
= \int_{\pd \MM} (\nabla \eta \cdot \nu) | \png w |^2\ \dwg
=- \gamma \int_{\pd \MM} \psi (\nabla \vartheta \cdot \nu) | \png w |^2\ \dwg.
\ee
Moreover it holds true that
\bea
& & \langle \nag w , \nag \langle \nag \eta , \nag \overline{w} \rg \rg = 
-\gamma \langle \nag w , \nag ( \psi \langle \nag \vartheta , \nag \overline{w} \rg ) \rg \nonumber \\
& = & - \gamma \psi \langle \nag w , \nag \langle \nag \vartheta , \nag \overline{w} \rg \rg 
- \gamma^2 \psi | \langle \nag \vartheta , \nag w \rg |^2. \label{ca8}
\eea
Further, as $\langle \nag w , \nag \langle \nag \vartheta , \nag \overline{w} \rg \rg  
= \nag w(\langle \nag \vartheta , \nag \overline{w} \rg)$, we have
\beas
& & \langle \nag w , \nag \langle \nag \vartheta , \nag \overline{w} \rg \rg 
 = \langle \mathbb{D}_{\nag w} \nag \vartheta , \nag \overline{w} \rg + \langle \nag \vartheta , \mathbb{D}_{\nag w} \nag \overline{w} \rg \\
& = & \mathbb{D}^2 \vartheta (\nag w,\nag \overline{w}) + \mathbb{D}^2 \overline{w} (\nag w,\nag \vartheta),
\eeas
so \eqref{ca6}--\eqref{ca8} yields
\bea
I_{2,2} & = & 2 s \gamma \int_{\MM_T} \psi \mathbb{D}^2 \vartheta (\nag w,\nag \overline{w})\ \dvg\ \dd t
+  2 s \gamma \Pre{\int_{\MM_T} \psi \mathbb{D}^2 \overline{w} (\nag w,\nag \vartheta)\ \dvg\ \dd t} \nonumber \\
& & + 2s \gamma^2 \int_{\MM_T} \psi | \langle \nag \vartheta , \nag w \rg |^2\ \dvg\ \dd t
-2s \gamma \int_{\pd \MM_T} \psi (\nabla \vartheta \cdot \nu) | \png w |^2\ \dwg\ \dd t. \label{ca9}
\eea
The next step involves noticing that 
$\langle \nag \eta, \nag \langle \nag w , \nag \overline{w} \rg \rg
= \nag \eta (\langle \nag w , \nag \overline{w} \rg)$ is equal to
$\langle \mathbb{D}_{\nag \eta} \nag w , \nag \overline{w} \rg + \langle \nag w , \mathbb{D}_{\nag \eta} \nag \overline{w} \rg$, in such a way that we have
\beas
& & \langle \nag \eta, \nag \langle \nag w , \nag \overline{w} \rg \rg
= \mathbb{D}^2 w(\nag \eta , \nag \overline{w}) + \mathbb{D}^2 \overline{w}(\nag \eta , \nag w) \\
& = & - \gamma \psi \left( \mathbb{D}^2 w(\nag \vartheta , \nag \overline{w}) + \mathbb{D}^2 \overline{w}(\nag \vartheta , \nag w) \right)
 = - 2\gamma \psi \Pre{\mathbb{D}^2 \overline{w}(\nag \vartheta , \nag w)},
\eeas
and consequently
\beas
& & 2 \gamma \Pre{\int_{\MM} \psi \mathbb{D}^2 \overline{w}(\nag w , \nag \vartheta)\ \dvg}= 
-\int_{\MM} \langle \nag \eta, \nag \langle \nag w , \nag \overline{w} \rg \rg\ \dvg \\
& = & -\int_{\MM} (\lag \eta) \langle \nag w , \nag \overline{w} \rg\ \dvg + \int_{\pd \MM} (\nabla \eta \cdot \nu) | \png \omega |^2\ \dwg \\
& = & -\gamma \int_{\MM} \psi (\lag \vartheta +\gamma | \nag \vartheta |_g^2 ) \langle \nag w , \nag \overline{w} \rg\ \dvg + \gamma \int_{\pd \MM} \psi (\nabla \vartheta \cdot \nu) | \png \omega |^2\ \dwg,
\eeas
by recalling that $\nag w \cdot \tau = 0$ on $\pd \MM$.
From this and \eqref{ca9} then follows that
\bea
I_{2,2} & = & 2 s \gamma \int_{\MM_T} \psi \mathbb{D}^2 \vartheta (\nag w,\nag \overline{w})\ \dvg\ \dd t
- s \gamma \int_{\MM_T} \psi ( \lag \vartheta + \gamma  | \nag \vartheta |_g^2 ) \langle \nag w , \nag \overline{w} \rg\
\dvg\ \dd t \nonumber \\
& & + 2s \gamma^2 \int_{\MM_T} \psi | \langle \nag \vartheta , \nag w \rg |^2\ \dvg\ \dd t
-s \gamma \int_{\pd \MM_T} \psi (\nabla \vartheta \cdot \nu) | \png w |^2\ \dwg\ \dd t. \label{ca10}
\eea
Let us now consider
\bea
& & I_{2,3}  :=  s \Pre{\int_{\MM_T} (\lag w) (\lag \eta) \overline{w}\ \dvg\ \dd t}
= -s \Pre{\int_{\MM_T} \langle \nag w , \nag ( (\lag \eta) \overline{w}) \rg\ \dvg\ \dd t} \nonumber \\
& = & -s  \Pre{\int_{\MM_T} \langle \nag w , \nag (\lag \eta) \rg \overline{w}\ \dvg\ \dd t}
-s \int_{\MM_T}  (\lag \eta) \langle \nag w , \nag \overline{w} \rg\ \dvg\ \dd t \nonumber \\
& = & s \gamma \int_{\MM_T} \psi ( \lag \vartheta + \gamma | \nag \vartheta |_g^2 ) \langle \nag w , \nag \overline{w} \rg\ \dvg\ \dd t  \nonumber \\
& + & s \gamma \Pre{\int_{\MM_T} \psi \langle \nag( \lag \vartheta + \gamma | \nag \vartheta |_g^2 ) , \nag \overline{w} \rg 
w\ \dvg \dd t}  \nonumber \\
& + & s \gamma^2 \Pre{ \int_{\MM_T} \psi ( \lag \vartheta + \gamma | \nag \vartheta |_g^2 ) \langle  \nag \vartheta , \nag \overline{w} \rg w\ \dvg \dd t}. \label{ca11}
\eea
The strategy now is to re-express each of the two last terms in the rhs of \eqref{ca11}. 
To do that we substitute $(1 \slash 2) \langle \nag \vartheta , \nag | w |^2 \rg$ for
$\Pre{\langle \nag \vartheta , \nag \overline{w} \rg w}$ in the last integral and apply \eqref{gGreen}, getting
\beas
& & \Pre{\int_{\MM_T} \psi ( \lag \vartheta + \gamma | \nag \vartheta |_g^2 ) \langle \nag \vartheta , \nag \overline{w} \rg w\ \dvg\ \dd t} \\
& = & \frac{1}{2} \int_{\MM_T} \langle \psi ( \lag \vartheta + \gamma | \nag \vartheta |_g^2 )
\nag \vartheta , \nag | w |^2 \rg\ \dvg\ \dd t \\
& = & - \frac{1}{2} \int_{\MM_T} \divg( \psi  ( \lag \vartheta + \gamma | \nag \vartheta |_g^2 ) \nag \vartheta ) | w |^2\ \dvg\ \dd t,
\eeas
since $w=0$ on $\pd \MM$. In light of \eqref{gdiv} this entails
\bea
& & \Pre{\gamma \int_{\MM} \psi ( \lag \vartheta + \gamma | \nag \vartheta |_g^2 ) \langle \nag \vartheta , \nag \overline{w} \rg w\ \dvg} \nonumber \\
& = & - \frac{1}{2} \int_{\MM} \psi \left( ( \lag \vartheta + \gamma | \nag \vartheta |_g^2 )^2 + \langle \nag  ( \lag \vartheta + \gamma | \nag \vartheta |_g^2 ) , \nag \vartheta \rg |w|^2 \right)\ \dvg. \label{ca12}
\eea
Similarly, by bringing the penultimate integral in the rhs of \eqref{ca11} into the form
\beas 
& & \Pre{\int_{\MM_T} \psi \langle  \nag( \lag \vartheta + \gamma | \nag \vartheta |_g^2 ) , \nag \overline{w}\rg w \dvg\ \dd t} \\
& = & \frac{1}{2} \int_{\MM_T} \psi \langle \nag ( \lag \vartheta + \gamma | \nag \vartheta |_g^2 ) , \nag | w |^2 \rg\ \dvg \dd t,
\eeas
and then applying \eqref{gGreen}, we find out that
\beas 
& &\Pre{\int_{\MM_T} \psi \langle  \nag( \lag \vartheta + \gamma | \nag \vartheta |_g^2 ) , \nag  \overline{w} \rg w\ \dvg \dd t} \\
& =&  -\frac{1}{2} \int_{\MM_T} \divg(\psi \nag  ( \lag \vartheta + \gamma | \nag \vartheta |_g^2 ) ) | w |^2\ \dvg\ \dd t \\
& = & -\frac{1}{2} \int_{\MM_T} \psi \left(\lag ( \lag \vartheta + \gamma | \nag \vartheta |_g^2 ) + \gamma \langle
\nag \vartheta , \nag ( \lag \vartheta + \gamma | \nag \vartheta |_g^2 ) \rg \right) | w |^2\ \dvg\ \dd t.
\eeas
Along with \eqref{ca11}-\eqref{ca12}, this yields
\bea
I_{2,3} & = & s \gamma  \int_{\MM_T} \psi ( \lag \vartheta + \gamma | \nag \vartheta |_g^2 ) \langle \nag w , \nag \overline{w} \rg\ \dvg\ \dd t \nonumber \\
& - &  \frac{s \gamma}{2} \int_{\MM_T} \psi \left[ \lag( \lag \vartheta + \gamma | \nag \vartheta |_g^2 ) + 2 \gamma \langle \nag \vartheta , \nag ( \lag \vartheta + \gamma | \nag \vartheta |_g^2 ) \rg \right. \nonumber \\
& & \hspace*{1.7cm} \left. + \gamma ( \lag \vartheta + \gamma | \nag \vartheta |_g^2 )^2 \right] | w |^2\ \dvg\ \dd t. \label{ca13}
\eea
Further, the case of 
\bel{ca14}
I_{3,1} := s^3 \Pre{\int_{\MM_T} | \nag \eta |_g^2 w (-\imath \eta' \overline{w})\ \dvg\ \dd t} = \Pim{\int_{\MM_T} | \nag \eta |_g^2 \eta' | w |^2 \ \dvg\ \dd t}=0,
\ee 
being easily treated, we turn our attention to the computation of
\beas
I_{3,2} & := & 2 s^3 \Pre{\int_{\MM_T} | \nag \eta |_g^2 \langle \nag \eta , \nag \overline{w} \rg w\ \dvg\ \dd t}
= \int_{\MM_T}  | \nag \eta |_g^2 \langle  \nag \eta , \nag |w|^2 \rg\ \dvg\ \dd t \\
& = & -\int_{\MM_T} \divg( | \nag \eta |_g^2 \nag \eta ) | w |^2\ \dvg\ \dd t.
\eeas
From this and the two identities $\divg( | \nag \eta |_g^2 \nag \eta )=| \nag \eta |_g^2 \lag \eta + \langle \nag \eta , \nag | \nag \eta |_g^2 \rg$ and $\langle \nag \eta , \nag | \nag \eta |_g^2 \rg = - \gamma^3 \psi^3( 2 \gamma | \nag \vartheta |_g^2 + \langle \nag \vartheta , \nag | \nag \vartheta |_g^2 \rg)$ then follows that
\bel{ca15}
I_{3,2} = s^3 \gamma^3 \int_{\MM_T} \psi^3 (3 \gamma | \nag \vartheta |_g^4 + \langle \nag \vartheta , \nag | \nag \vartheta |_g^2 \rg + | \nag \vartheta |_g^2 \lag \vartheta )\ \dvg\ \dd t.
\ee
Finally, we get
\bel{ca16} 
I_{3,3} := s^2 \int_{\MM_T} | \nag \eta |_g^2 (\lag \eta) | w |^2\ \dvg\ \dd t = - s^3 \gamma^3 \int_{\MM_T} \psi^3 ( \gamma | \nabla \vartheta |_g^4 + | \nabla \vartheta |_g^2 \lag \vartheta) |w|^2\ \dvg\ \dd t,
\ee
by straighforward computations. Now, putting \eqref{ca0}--\eqref{ca1}, \eqref{ca5b}, \eqref{ca10} and \eqref{ca13}--\eqref{ca16} together, we have obtained that
\bea
& & \Pre{\int_{\MM_T} M_1 w \overline{M_2 w}\ \dvg\ \dd t} \\
& = & 2 s \gamma^2 \int_{\MM_T} \psi | \langle \nag \vartheta , \nag w \rg |^2\ \dvg\ \dd t 
+ 2 s \gamma \int_{\MM_T} \psi \mathbb{D}^2 \vartheta(\nag w , \nag \overline{w})\ \dvg\ \dd t \nonumber \\
& + & 2 s^3 \gamma^4 \int_{\MM_T} \psi^3 | \nag \vartheta |_g^4 | w |^2\ \dvg\ \dd t
- s \gamma \int_{\pd \MM_T} (\nabla \vartheta \cdot \nu) | \png w |^2\ \dwg\ \dd t - \mathcal{R}, \label{ca17}
\eea
where the remaining term $\mathcal{R}$ collects the integrals which are negligible wrt to either 
$s \gamma^2 \int_{\MM_T} \psi | \langle \nag \vartheta , \nag w \rg |^2\ \dvg\ \dd t$ or $s^3 \gamma^4 \int_{\MM_T} \psi^3 | \nag \vartheta |_g^4 | w |^2\ \dvg\ \dd t$:
\bea
\mathcal{R} & := & \frac{s}{2} \int_{\MM_T} \eta'' |w|^2\ \dvg\ \dd t + 2 s \gamma \Pim{\int_{\MM_T} \psi' \langle  \nag \vartheta , \nag \overline{w} \rg w\ \dvg\ \dd t} \nonumber \\
& + &  \frac{s \gamma}{2} \int_{\MM_T} \psi (\lag^2 \vartheta) | w |^2\ \dvg\ \dd t \nonumber \\
& + &  \frac{s \gamma^2}{2} \int_{\MM_T} \psi \left[ (\lag \vartheta)^2 + 2 \langle \nag \vartheta , \nag(\lag \vartheta) \rg + \lag | \nag \vartheta |_g^2 \right] |w|^2\ \dvg\ \dd t \nonumber \\
& + &  s \gamma^3 \int_{\MM_T} \psi \left[ (\lag \vartheta) | \nag \vartheta |_g^2  + \langle \nag \vartheta , \nag(\lag \vartheta) \rg  \right] |w|^2\ \dvg\ \dd t \nonumber \\
& + &  \frac{s \gamma^4}{2} \int_{\MM_T} \psi  | \nag \vartheta |_g^4 |w|^2\ \dvg\ \dd t \nonumber\\
& - &  s^3 \gamma^3 \int_{\MM_T} \psi^3 \langle \nag \vartheta , \nag | \nag \vartheta |_g^2 \rg  |w|^2\ \dvg\ \dd t.
\label{ca18}
\eea
Since 
$$\left| \int_{\MM_T} \psi' \langle \nag \vartheta , \nag \overline{w} \rg w\ \dvg\ \dd t \right| \leq \int_{\MM_T} | \psi'|^{1 \slash 2}
| \langle \nag \vartheta , \nag w \rg |^2\ \dvg\ \dd t + \int_{\MM_T} | \psi' |^{3 \slash 2} | w |^2\ \dvg\ \dd t,$$ 
by Young inequality and 
$$ 0 \leq  \psi(t,x)  \leq C \psi(t,x)^3,\ | \psi'(t,x) | \leq C \psi(t,x)^2,\ | \eta''(t,x) | \leq C \psi(t,x)^3,\ (t,x) \in \MM_T, $$
for some constant $C=C(T)>0$, depending only on $T$, by \eqref{c3}, 
we deduce from \eqref{ca18} that
\bel{ca19}
| \mathcal{R} | \leq C \left( s \gamma \int_{\MM_T} \psi | \langle \nag \vartheta , \nag w \rg |_g^2\ \dvg\ \dd t + s^3 \gamma^3 \int_{\MM_T} \psi^3 | w |^2\ \dvg\ \dd t + s \gamma^4 \int_{\MM_T} \psi |w|^2\ \dvg\ \dd t \right),
\ee
whenever $s \geq 1$ and $\gamma \geq 1$. Here $C>0$ denotes some positive constant which depends only on $T$. 

The last step of the proof involves putting the following basic identity
$$\int_{\MM_T} | M w |^2\ \dvg\ \dd t = \int_{\MM_T} (| M_1 w |^2 + | M_2 w|^2)\ \dvg\ \dd t + 2 \Pre{\int_{\MM_T} M_1 w \overline{M_2 w}\ \dvg\ \dd t},$$
and \eqref{ca17} together, getting 
\beas
& & 4 s \gamma^2 \int_{\MM_T} \psi | \langle \nag \vartheta , \nag w \rg |^2\ \dvg\ \dd t +
\int_{\MM_T} ( |M_1 w |^2 + |M_1 w |^2)\ \dvg\ \dd t \\
& + &  4 s \gamma \int_{\MM_T} \psi \mathbb{D}^2 \vartheta(\nag w , \nag \overline{w})\ \dvg\ \dd t + 4 s^3 \gamma^4  \int_{\MM_T} \psi^3 | \nag \vartheta |_g^4| w |^2\ \dvg\ \dd t \\
& \leq & \int_{\MM_T} | M w |^2\ \dvg\ \dd t + 2 s \gamma \int_{\pd \MM_T} \psi (\nabla \vartheta \cdot \nu) | \png w |^2\ \dwg\ \dd t + 2 | \mathcal{R} |,
\eeas
and then using \eqref{caH2b} and \eqref{ca19}. We obtain that
\beas
& & s \gamma^2 \int_{\MM_T} \psi | \langle \nag \vartheta , \nag w \rg |^2\ \dvg\ \dd t +
\int_{\MM_T} ( |M_1 w |^2 + |M_1 w |^2)\ \dvg\ \dd t \\
& + &  s \gamma \int_{\MM_T} \psi \mathbb{D}^2 \vartheta(\nag w , \nag \overline{w})\ \dvg\ \dd t + s^3 \gamma^4  \int_{\MM_T} \psi^3 | w |^2\ \dvg\ \dd t \\
& \leq & C \left( \int_{\MM_T} | M w |^2\ \dvg\ \dd t + 2 s \gamma \int_{\pd \MM_T} \psi (\nabla \vartheta \cdot \nu) | \png w |^2\ \dwg\ \dd t \right. \\
& + & \left. s \gamma  \int_{\MM_T} | \langle \nag \vartheta , \nag w \rg |_g^2\ \dvg\ \dd t + s^3 \gamma^3 \int_{\MM_T} \psi^3 | w |^2\ \dvg\ \dd t + s \gamma^4 \int_{\MM_T} \psi |w|^2\ \dvg\ \dd t \right),
\eeas
for every $s \geq 1$ and $\gamma \geq 1$, the constant $C=C(T,\beta)$ depending only on $T$ and $\beta$. This easily entails
\beas
& & s \gamma^2 \int_{\MM_T} \psi | \langle \nag \vartheta , \nag w \rg |^2\ \dvg\ \dd t +
\int_{\MM_T} ( |M_1 w |^2 + |M_1 w |^2)\ \dvg\ \dd t \\
& + & s \gamma \int_{\MM_T} \psi \mathbb{D}^2 \vartheta(\nag w , \nag \overline{w})\ \dvg\ \dd t + s^3 \gamma^4 \int_{\MM_T} \psi^3 | w |^2\ \dvg\ \dd t \\
& \leq & C \left( \int_{\MM_T} | M w |^2\ \dvg\ \dd t + s \gamma \int_{\pd \MM_T} \psi (\nabla \vartheta \cdot \nu) | \png w |^2\ \dwg\ \dd t \right),
\eeas
upon taking $s$ and $\gamma$ sufficiently large, so the desired result follows immediately from this, \eqref{caH1b} and \eqref{caH3}.

\end{document}